\documentclass[12pt, oneside, psamsfonts]{amsart}

\newif\ifPDF
\ifx\pdfoutput\undefined\PDFfalse
\else \ifnum \pdfoutput > 0 \PDFtrue
        \else \PDFfalse
        \fi
\fi

\usepackage[centertags]{amsmath}
\usepackage{amsfonts}
\usepackage{mathrsfs}
\usepackage{textcomp}
\usepackage{amssymb}
\usepackage{amsthm}
\usepackage{newlfont}
\usepackage[all]{xy}

\ifPDF
  \usepackage[pdftex]{color, graphicx}
  \usepackage[pdftex, bookmarks, colorlinks]{hyperref}
  \hypersetup{colorlinks=false}

\else
  \usepackage{color}
  \usepackage[dvips]{graphicx}
  \usepackage[dvips]{hyperref}
\fi

\usepackage[scale=0.8]{geometry}

\usepackage[pagewise, mathlines, displaymath]{lineno}

\newtheorem{thm}{Theorem}[section]
\newtheorem{cor}[thm]{Corollary}
\newtheorem{lem}[thm]{Lemma}
\newtheorem{prop}[thm]{Proposition}
\theoremstyle{definition}
\newtheorem{defn}[thm]{Definition}
\theoremstyle{remark}
\newtheorem{rem}[thm]{Remark}

\numberwithin{equation}{section}

\newcommand{\norm}[1]{\left\Vert#1\right\Vert}
\newcommand{\abs}[1]{\left\vert#1\right\vert}

\newcommand{\Real}{\mathbb R}
\newcommand{\Int}{\mathbb Z}
\newcommand{\Comp}{\mathbb C}

\newcommand{\eps}{\varepsilon}

\newcommand{\F}{\mathcal{F}}
\newcommand{\Kzero}{\textrm{K}_0}

\begin{document}


\title{The C*-algebra of a minimal homeomorphism of zero mean dimension}

\author{George A. Elliott}
\address{Department of Mathematics, University of Toronto, Toronto, Ontario, Canada~\ M5S 2E4}
\email{elliott@math.toronto.edu}

\author{Zhuang Niu}
\address{Department of Mathematics, University of Wyoming, Laramie, WY, 82071, USA.}
\email{zniu@uwyo.edu}



\begin{abstract}
Let $X$ be an infinite compact metrizable space, and let $\sigma: X\to X$ be a minimal homeomorphism. Suppose that $(X, \sigma)$ has zero mean topological dimension. The associated C*-algebra $A=\mathrm{C}(X)\rtimes_\sigma\mathbb Z$ is shown to absorb the Jiang-Su algebra $\mathcal Z$ tensorially, i.e., $A\cong A\otimes\mathcal Z$. This implies that $A$ is classifiable when $(X, \sigma)$ is uniquely ergodic. 

Moreover, without any assumption on the mean dimension, it is shown that $A\otimes A$ always absorbs the Jiang-Su algebra.
\end{abstract}

\maketitle

\setcounter{tocdepth}{1}

\section{Introduction}

Recently, Toms and Winter proved that a simple C*-algebra arising from an action of the group $\Int$ of integers on a  metrizable compact space of finite dimension absorbs the Jiang-Su algebra $\mathcal Z$ (\cite{TW-Dym}, \cite{TW-Dym-1}). (This definitive result followed much earlier work, e.g., \cite{LP-Dym}.) As shown in \cite{GK-Dyn}, some condition on the dynamical system is necessary. Possibly, the condition of mean dimension zero, which we shall show is sufficient, is also necessary. (Phillips and Toms have conjectured that the mean dimension of a minimal dynamical system is always exactly twice the radius of comparison of the associated (crossed product) C*-algebra.
)

In the present note, we shall show that the condition of mean dimension zero (reviewed in Definition \ref{MD-def} below) is sufficient: it implies that the C*-algebra of the dynamical system absorbs the Jiang-Su C*-algebra $\mathcal Z$. 
More precisely, one has
\theoremstyle{plain}
\newtheorem*{thm-nocount}{Theorem}
\begin{thm-nocount}
Let $X$ be an infinite compact metrizable space, and let $\sigma: X\to X$ be a minimal homeomorphism. If $(X, \sigma)$ has mean dimension zero, then the crossed product C*-algebra $A=\mathrm{C}(X)\rtimes_{\sigma}\mathbb Z$ absorbs the Jiang-Su algebra $\mathcal Z$ tensorially.
\end{thm-nocount}

The same classification consequences as shown in \cite{TW-Dym} and \cite{TW-Dym-1} in the case that $\Kzero$ separates traces hold also in the present setting; in particular, the C*-algebra of any uniquely ergodic dynamical system is classifiable (in this case the mean dimension is automatically zero). 

Together with recent results in the classification of C*-algebras (see \cite{ENST-ASH}, \cite{GLN-TAS}, \cite{Lin-Dyn}, \cite{EGLN-ASH}, \cite{EN-K0-Z} and \cite{EGLN-Dr}), this theorem implies that the C*-algebra of a minimal dynamical system with mean dimension zero is always classifiable. In particular, this includes all the dynamical systems with finite entropy or the systems with countably many ergodic measures (\cite{Lindenstrauss-Weiss-MD}, \cite{Lind-MD}).

The proof of the main theorem also shows that the tensor product of the C*-algebras of two arbitrary minimal homeomorphisms (without any assumption on the mean dimension) is, perhaps surprisingly, always Jiang-Su stable:
\begin{thm-nocount}
Let $(X_1, \sigma_1)$ and $(X_2, \sigma_2)$ be minimal dynamical systems, where $X_1$ and $X_2$ are infinite compact metrizable spaces. Consider the C*-algebras $$A_1=\mathrm{C}(X_1)\rtimes_{\sigma_1}\Int \quad \textrm{and}\quad A_{2}= \mathrm{C}(X_2)\rtimes_{\sigma_2}\Int.$$
Then $$A_{1}\otimes A_{2}\cong (A_{1}\otimes A_{2})\otimes\mathcal Z.$$
\end{thm-nocount}

Again, with the recent classification results referred to above, this implies that $A_1\otimes A_2$ is always a classifiable C*-algebra.


\section{Mean topological dimension and the small boundary property}
Let $X$ be a compact metrizable space, and let $\sigma: X\to X$ be a homeomorphism. (These objects will be fixed throughout the paper.)
\begin{defn}[\cite{Lindenstrauss-Weiss-MD}]\label{MD-def}
The mean (topological) dimension of $(X, \sigma)$, denoted by $\mathrm{mdim}(X, \sigma)$,  is defined by
$$\mathrm{mdim}(X, \sigma)=\sup_{\alpha}\lim_{N\to\infty}\frac{1}{N}\mathcal{D}(\alpha\vee\sigma(\alpha)\vee\cdots\vee\sigma^{N-1}(\alpha)),$$
where the supremum is taken over arbitrary finite open covers $\alpha$, and the dimension of a finite open cover $\beta$, $\mathcal D(\beta)$, is the number $\min\{\mathrm{ord}(\beta');\ \beta'\preceq \beta\}$. (By the order of a cover $\beta$ is meant the number $\mathrm{ord}(\beta)=-1+\sup_x\sum_{U\in\beta} \chi_U(x)$, and by the join $\beta_1\vee\beta_2$ of two covers $\beta_1$ and $\beta_2$ is meant the cover $\{U_1\cap U_2: U_1\in\beta_1, U_2\in\beta_2\}$.)
\end{defn}

\begin{defn}[\cite{Lindenstrauss-Weiss-MD}]
For each set $E\subseteq X$, the orbit capacity of $E$, denoted by $\mathrm{ocap}(E)$, is defined to be
$$\mathrm{ocap}(E)=\lim_{N\to\infty}\frac{1}{N}\sup \{\chi_{E}(x)+\chi_{E}(\sigma(x))+\cdots+ \chi_{E}(\sigma^{N-1}(x));\ x\in X\}.$$
The system $(X, \sigma)$ is said to have the small boundary property (SBP) if for any $x\in X$ and any open neighbourhood $U$ of $x$, there is a neighbourhood $V$ in $U$ such that $\mathrm{ocap}(\partial V)=0$.
\end{defn}

\begin{thm}[Theorem 5.4 of \cite{Lindenstrauss-Weiss-MD} and Theorem 6.2 of \cite{Lind-MD}]
If $\sigma$ is minimal, then $(X, \sigma)$ has zero mean dimension if and only if it has the small boundary property.
\end{thm}

\begin{prop}[Proposition 5.3 of \cite{Lindenstrauss-Weiss-MD}]\label{small-intsct}
If $(X,\sigma)$ has the SBP, then for every finite open cover $\alpha$ of $X$ and every $\eps>0$, there is a partition of unity $\phi_j: X\to [0, 1]$ ($j=1, ..., \abs{\alpha}$) subordinate to $\alpha$
such that
$$\mathrm{ocap}(\bigcup_{j=1}^{\abs{\alpha}}\phi_j^{-1}((0, 1)))<\eps.$$
\end{prop}

\section{The C*-algebra of a homeomorphism and its large subalgebras}\label{Strut-AY}
Suppose that $X$ as above is an infinite set and $\sigma$ as above is minimal.  Let us denote by $\sigma$ also the automorphism of $\mathrm{C}(X)$ defined by $$\sigma(f)=f\circ\sigma^{-1},\quad f\in \mathrm{C}(X).$$ 
Consider the crossed product C*-algebra
$$A=\mathrm{C}(X)\rtimes_\sigma\Int=\textrm{C*}\{f, u;\ ufu^*=\sigma(f),\ f\in \mathrm{C}(X)\}.$$
The C*-algebra $A$ is nuclear. Since $\sigma$ is minimal, the C*-algebra $A$ is simple.
Fix $y\in X$, and consider the Putnam algebra (\cite{Put-PJM}) 
$$A_y=\textrm{C*}\{f, ug;\ f, g\in\mathrm{C}(X),\ g(y)=0\}\subseteq A.$$
Let $Y$ be a closed neighbourhood of $y$ in $X$. Consider the sub-C*-algebra
$$A_Y=\textrm{C*}\{f, ug;\ f, g\in\mathrm{C}(X),\ g|_Y=0\}\subseteq A_y.$$
It follows from the definition that $A_{Y_1}\subseteq A_{Y_2}$ if $Y_1\supseteq Y_2$, and $A_y=\overline{\bigcup_{i=1}^\infty A_{Y_i}}$ 
if $\bigcap Y_i=\{y\}$.

Consider the first return times
$$\{j\in \mathbb N\cup \{0\};\ \textrm{$\sigma^j(x)\in Y$, $\sigma^i(x)\notin Y$, $1\leq i\leq j-1$ for some $x\in Y$}\}.$$
Since $\sigma$ is minimal and $X$ is compact, this set of numbers is finite; let us write it as
$$J_1<J_2<\cdots <J_K$$
for some $K\in\mathbb N$.
Since $X$ is an infinite set and $\sigma$ is minimal, the first return time $J_1$ is arbitrarily large if $Y$ is sufficiently small.

For each $1\leq k\leq K$, consider the (locally compact---see below) subset of $X$
$$Z_k=\{x\in Y;\  \textrm{$\sigma^{J_k}(x)\in Y$ but $\sigma^i(x)\notin Y$ for any $1\leq i\leq J_k-1$}\}.$$
Then the sets 
$$Z_1, \sigma(Z_1), ..., \sigma^{J_1-1}(Z_1), ..., Z_k, \sigma(Z_k), ..., \sigma^{J_k-1}(Z_k)$$
---which are naturally listed as shown---form a partition of $X$. This is often called a Rokhlin partition.

\begin{lem}[\cite{Lin-Q-RSA}; see also \cite{QL-Ph-min-diff}]
In terms of the notation introduced above, one has that, for each $1\leq k\leq K$,\\
\indent $\mathrm{(1)}$ the set $Z_1\cup\cdots\cup Z_k$ is closed (and so $Z_k$ is locally compact),\\
\indent $\mathrm{(2)}$ the set $\overline{Z_k}\cap (Z_1\cup\cdots\cup Z_{k-1})$ is the disjoint union of the subsets
$$W_{t_1,..., t_s}=\partial Z_k\cap Z_{t_1}\cap\sigma^{-J_{t_1}}(Z_{t_2})\cap\cdots\cap\sigma^{-(J_{t_1}+\cdots+J_{t_{s-1}})}(Z_{t_s}),$$ where $J_{t_1}+\cdots +J_{t_{s-1}}+J_{t_s}=J_k$.
\end{lem}

A quite explicit description of the subalgebra $A_Y$ of the crossed product, a C*-algebra of type I,  was obtained by Q.~Lin (\cite{Lin-Q-RSA}). It is a subhomogeneous algebra, of order at most $J_K$. (In other words, all its irreducible representations have at most this dimension.)

\begin{thm}[\cite{Lin-Q-RSA}; see also \cite{QL-Ph-min-diff}]\label{Lin-Sub}
In terms of the notation introduced above, one has that the C*-algebra $A_Y$ is isomorphic to the sub-C*-algebra of $\bigoplus_{k=1}^K \mathrm{M}_{J_k}(\mathrm{C}(\overline{Z_k}))$ consisting of the elements $(F_1, ..., F_K)$ with
\begin{displaymath}
F_k|_{W_{t_1, ..., t_s}}=
\left(
\begin{array}{cccc}
F_{t_1}|_{W_{t_1, ..., t_s}} & & & \\
& F_{t_2}\circ\sigma^{J_{t_1}}|_{W_{t_1, ..., t_s}} & &\\
& & \ddots &\\
& & & F_{t_s}\circ\sigma^{J_{t_s-1}}|_{W_{t_1, ..., t_s}}
\end{array}
\right)
\end{displaymath}
whenever 
$$W_{t_1, ..., t_s}=\partial Z_k\cap Z_{t_1}\cap\sigma^{-J_{t_1}}(Z_{t_2})\cap\cdots\cap\sigma^{-(J_{t_1}+\cdots+J_{t_{s-1}})}(Z_{t_s})\neq\O,$$ where $J_{t_1}+\cdots +J_{t_{s-1}}+J_{t_s}=J_k$. 

Moreover, for any $f, g\in\mathrm{C}(X)$ with $g|_Y=0$, the images of $f, ug\in A_Y$ in this identification are
\begin{equation}\label{id-01}
f=\bigoplus_{k=1}^K
\left(
\begin{array}{cccc}
f\circ \sigma|_{\overline{Z_k}} & & & \\
& f\circ \sigma^2|_{\overline{Z_k}} & &\\
& & \ddots &\\
& & & f\circ \sigma^{J_k}|_{\overline{Z_k}}
\end{array}
\right)\in \bigoplus_{k=1}^K \mathrm{M}_{J_k}(\mathrm{C}(\overline{Z_k}))
\end{equation}
and
\begin{equation}\label{id-02}
ug=\bigoplus_{k=1}^K
\left(
\begin{array}{cccc}
0 & & & \\
g\circ \sigma|_{\overline{Z_k}}& 0& & \\
&  \ddots& \ddots &\\
& & g\circ \sigma^{J_k-1}|_{\overline{Z_k}} & 0 
\end{array}
\right)\in \bigoplus_{k=1}^K \mathrm{M}_{J_k}(\mathrm{C}(\overline{Z_k})),
\end{equation}
respectively.
\end{thm}

The sub-C*-algebra $A_y$ of $A$ (with $y\in X$) is a typical example of a large sub-algebra:
\begin{defn}[\cite{NCP-LAlg}, \cite{A-NCP-LAlg}]
Let $A$ be an infinite dimensional unital simple separable C*-algebra. A unital sub-C*-subalgebra $B\subseteq A$ is said to be large in $A$ if for every
$m\in \mathbb{Z}_{>0}$, $a_1, a_2, ..., a_m\in A$, $\eps>0$, $x\in A^+$ with $\norm{x}=1$, and $y\in B^+\setminus\{0\}$, there are $c_1, c_2, ..., c_m\in A$ and $g\in B$ such that:
\begin{itemize}
\item[(1)] $0 \leq g \leq 1$,
\item[(2)] For $j=1, 2, ..., m$ we have $\norm{c_j-a_j}<\eps$,
\item[(3)] For $j=1, 2, ..., m$ we have $(1-g)c_j, c_j(1-g)\in B$,
\item[(4)] $g\preceq_B y$, and
\item[(5)] $\norm{(1-g)x(1-g)}>1-\eps$.
\end{itemize}
If, moreover, the element $g$ can be chosen so that
\begin{itemize}
\item[(6)] $\norm{ga_j-a_jg}<\eps$, $j=1, 2, ..., m$,
\end{itemize}
then the sub-C*-algebra $B$ is said to be centrally large in $A$.
\end{defn}

\begin{thm}[Theorem 4.6 of \cite{A-NCP-LAlg}]\label{LALG}
The C*-algebra $A_y$ is centrally large in $A$.
\end{thm}

\begin{thm}[Proposition 3.7 of \cite{A-NCP-LAlg}]\label{CLalg-prod}
Let $A_1$ and $A_2$ be infinite dimensional simple unital C*-algebras, and
let $B_1 \subset A_1$ and $B_2 \subset A_2$ be centrally large subalgebras.
Assume that $A_1 \otimes_{\min} A_2$ is finite. Then $B_1 \otimes_{\min} B_2$
is a centrally large subalgebra of $A_1 \otimes_{\min} A_2$.
\end{thm}

We will use the following property of centrally large sub-C*-algebras.
\begin{thm}[Theorem 2.3 of \cite{A-NCP-LAlg-Z}]\label{A-NCP-Z}
Let $A$ be a simple separable infinite dimensional nuclear unital C*-algebra, and let $B$ be a centrally large subalgebra of $A$. If $\mathcal Z\otimes B\cong B$, then $\mathcal Z\otimes A\cong A$.
\end{thm}

\section{The C*-algebra of a minimal homeomorphism of mean dimension zero }

Recall that a C*-algebra is said to be subhomogeneous if the dimensions of its irreducible representations are finite and uniformly bounded. Let $S$ be a subhomogeneous C*-algebra, with dimensions of irreducible representations $d_1<d_2<\cdots<d_n$. The dimension ratio of $S$ is defined as
$$\mathrm{dimRatio}(S):=\max\{\frac{\mathrm{dim}(\mathrm{Prim}_{d_1}(S))}{d_1}, \frac{\mathrm{dim}(\mathrm{Prim}_{d_2}(S))}{d_2}, ..., \frac{\mathrm{dim}(\mathrm{Prim}_{d_n}(S))}{d_n}  \},$$
where $\mathrm{Prim}_{d_i}(S)$ denotes the space of the primitive ideals of $A$ corresponding to the irreducible representations with dimension $d_i$, with the (relative) Jacobson topology ($\mathrm{Prim}_{d_i}(S)$ is then locally compact and Hausdorff; see Proposition 3.6.4(i) of \cite{Dix-Book}), and $\mathrm{dim}(\cdot)$ denotes the topological covering dimension. 

By Proposition 2.13 (together with 2.5 and 2.9) of \cite{Phill-RSA1}, if the primitive ideal spaces of $S$ have finite dimension, then the C*-algebra $S$ has a recursive subhomogeneous decomposition,
$$S\cong\left[\cdots\left[\left[C_0\oplus_{C_1^{(0)}}C_1\right]\oplus_{C_2^{(0)}}C_2\right]\cdots\right]\oplus_{C_l^{(0)}}C_l,$$ 
with $C_k=\mathrm{C}(X_k, \mathrm{M}_{n(k)})$ for a compact Hausdorff space $X_k$ and a positive integer $n(k)$, and with $C_k^{(0)}=\mathrm{C}(X_k^{(0)}, \mathrm{M}_{n(k)})$ for a compact subset $X_k^{(0)}\subseteq X_k$ (possibly empty), and then
$$\max\{\frac{\mathrm{dim}(X_k)}{n(k)}: 0\leq k\leq l\} = \mathrm{dimRatio}(S).$$
(See \cite{Phill-RSA1} for more details concerning recursive subhomogeneous C*-algebras.)

In the present section, it will be shown (using Theorem \ref{Lin-Sub} indirectly) that if $(X, \sigma)$ has zero mean dimension (and, as is understood, $\sigma$ is minimal), then the large subalgebra $A_y$ can be locally approximated by subhomogeneous C*-algebras with arbitrarily small dimension ratio (see Theorem \ref{lc-app}).

As a consequence of this, it follows (on applying the large subalgebra technique---see \cite{NCP-LAlg}) that the crossed product C*-algebra $\mathrm{C}(X)\rtimes_\sigma\Int$ absorbs the Jiang-Su algebra $\mathcal Z$, the main result of this paper (Theorem \ref{Z-stb}).

Of the following three lemmas (Lemmas \ref{small-nbhd}, \ref{uni-sep}, and \ref{cont-field}), only the first concerns dynamical systems; the other two are elementary C*-algebra results.

\begin{lem}\label{small-nbhd}
Let $Y\subseteq X$ be a closed subset with nonempty interior. Denote by $Z_1, Z_2, ..., Z_K$ the bases of the Rokhlin towers generated by $Y$, and by $J_1<J_2<\cdots<J_K$ the first return times of $Z_1, Z_2, ..., Z_K$, respectively. There is an open set $U\supseteq Y$ such that for each $1\leq k\leq K$, one has
$$\frac{1}{J_k}(\chi_{U}(x)+\chi_U(\sigma(x))+\cdots+ \chi_U(\sigma^{J_k-1}(x)))\leq \frac{1}{J_1},\quad x\in Z_{k}.$$
\end{lem}
\begin{proof}
Note that by definition the inequality holds with $Y$ in place of $U$. (So, the question is to extend this in some sense by continuity to a neighbourhood---we propose to do this by induction on $k$.)

Since $Y$ is closed, and the sets
$$Y, \sigma(Y), ..., \sigma^{J_1-1}(Y)$$
are pairwise disjoint, there is an open set $U\supseteq Y$ such that 
$$U, \sigma(U), ..., \sigma^{J_1-1}(U)$$
are pairwise disjoint. In particular, 
$$\frac{1}{J_1}(\chi_{U}(x)+\chi_U(\sigma(x))+\cdots+ \chi_U(\sigma^{J_1-1}(x)))\leq \frac{1}{J_1},\quad x\in Z_{1}.$$

Let $2\leq k\leq K$, and assume that we have constructed an open set $U\supseteq Y$ such that 
for any $1\leq i\leq k-1$, 
\begin{equation}\label{ind-asptn}
\frac{1}{J_i}(\chi_{U}(x)+\chi_U(\sigma(x))+\cdots+ \chi_U(\sigma^{J_i-1}(x)))\leq \frac{1}{J_1},\quad x\in Z_{i}.
\end{equation}
Let us construct another open neighbourhood of $Y$, still to be denoted by $U$ (just shrink!), such that \eqref{ind-asptn} holds for $i=k$.

First, pick an open neighbourhood $U'$ of $Y$ such that $\overline{U'}\subseteq U$. Let  
$x\in \overline{Z_{k}}\cap (Z_{1}\cup\cdots\cup Z_{k-1})$. If
$$x\in W_{t_1, ..., t_s}= \overline{Z_{k}}\cap Z_{t_1}\cap\sigma^{-J_{t_1}}(Z_{t_2})\cap\cdots\cap \sigma^{-(J_{t_1}+\cdots+J_{t_{s-1}})}(Z_{t_s}),$$
where $J_{t_1}+\cdots+J_{t_{s-1}}+J_{t_s}=J_k,$
then the orbit of $x$ is
$$\underbrace{x, \sigma(x),...,\sigma^{J_{t_1}-1}(x)}_{\textrm{in tower $Z_{t_1}$}}, \underbrace{\sigma^{J_{t_1}}(x), ..., \sigma^{J_{t_2}}(\sigma^{J_{t_1}}(x))}_{\textrm{in tower $Z_{t_2}$}}, ...,  \underbrace{\sigma^{J_{t_1}+\cdots+J_{t_{s-1}}}(x), ..., \sigma^{J_{t_s}}(\sigma^{J_{t_1}+\cdots+J_{t_{s-1}}}(x))}_{\textrm{in tower $Z_{t_s}$}}.$$
By the induction hypothesis \eqref{ind-asptn}, one has (note that $J_{t_1}+\cdots+J_{t_{s-1}}+J_{t_s}=J_k$)
$$
\frac{1}{J_k}(\chi_{U}(x)+\chi_U(\sigma(x))+\cdots+ \chi_U(\sigma^{J_k-1}(x)))\leq \frac{1}{J_1},
$$
and therefore, there is a neighbourhood $V_x$ of $x$ such that
\begin{equation}
\frac{1}{J_k}(\chi_{U'}(z)+\chi_{U'}(\sigma(z))+\cdots+ \chi_{U'}(\sigma^{J_k-1}(z)))\leq \frac{1}{J_1},\quad z\in V_x.
\end{equation}
Hence, there is an open set $E$ such that $$\overline{Z_{k}}\cap (Z_{1}\cup\cdots\cup Z_{k-1})\subseteq E$$ and 
\begin{equation}\label{small-bd}
\frac{1}{J_k}(\chi_{U'}(z)+\chi_{U'}(\sigma(z))+\cdots+ \chi_{U'}(\sigma^{J_k-1}(z)))\leq \frac{1}{J_1},\quad z\in E.
\end{equation}
Replace $U$ by $U'$ and still denote it by $U$. Since $Z_{1}\cup\cdots\cup Z_{k-1}\cup Z_{k}$ is a closed set, one has that
$$\overline{Z_{k}}\setminus Z_{k}\subseteq \overline{Z_{k}}\cap (Z_{1}\cup\cdots\cup Z_{k-1}),$$
and hence $\overline{Z_{k}}\setminus Z_{k}\subseteq E.$
In particular, 
$$\overline{Z_{k}}\setminus E = Z_{k}\setminus E ,$$
and $Z_{k}\setminus E$ is a compact set.

For any point $x$ in $Z_{k}\setminus E$, one can shrink $U$ further so that
\begin{equation}\label{small-comp-loc}
\frac{1}{J_k}(\chi_{\overline{U}}(x)+\chi_{\overline{U}}(\sigma(x))+\cdots+ \chi_{\overline{U}}(\sigma^{J_k-1}(x)))\leq \frac{1}{J_1}.
\end{equation}
Note that \eqref{small-comp-loc} holds for a neighbourhood of $x$. Since $Z_{k}\setminus E$ is compact, there is an open neighbourhood $U$ of $Y$ such that
\begin{equation}\label{small-comp}
\frac{1}{J_k}(\chi_{U}(x)+\chi_{U}(\sigma(x))+\cdots+ \chi_{U}(\sigma^{J_k-1}(x)))\leq \frac{1}{J_1},\quad x\in Z_{k}\setminus E.
\end{equation}
Together with \eqref{small-bd}, one has 
\begin{equation}
\frac{1}{J_k}(\chi_{U}(x)+\chi_{U}(\sigma(x))+\cdots+ \chi_{U}(\sigma^{J_k-1}(x)))\leq \frac{1}{J_1},\quad x\in Z_{k},
\end{equation}
as desired.
\end{proof}

\begin{lem}\label{uni-sep}
Consider $n\times n$ matrices with complex entries
$$A:=\mathrm{diag}(a_1, ..., a_n),\quad B:=\mathrm{diag}(b_1, ..., b_n)$$
\begin{displaymath}
C:=\left(
\begin{array}{cccc}
0 &    &  & \\
c_1 & 0  &  &  \\
 & \ddots  & \ddots & \\
 &    & c_{n-1} & 0
\end{array}
\right)
\quad
\textrm{and}
\quad
D:=\left(
\begin{array}{cccc}
0 &    &  & \\
d_1 & 0  &  &  \\
 & \ddots  & \ddots & \\
 &    & d_{n-1} & 0
\end{array}
\right),
\end{displaymath}
where $0<c_i, d_i\leq 1$.  If the pair $(A, C)$ is unitarily equivalent to the pair $(B, D)$, then
$$a_i=b_i,\ c_j=d_j,\quad 1\leq i\leq n,\  1\leq j\leq n-1 .$$
\end{lem}
\begin{proof}
Let $W\in\mathrm{M}_n(\Comp)$ be a unitary such that
$$W^*AW=B\quad\textrm{and}\quad W^*CW=D.$$  
For each $1\leq k\leq n$, one has $W^*((C^*)^kC^k)W=(D^*)^kD^k$, and 
a functional calculus argument shows that 
$$W^*(e_1+\cdots +e_k)W=e_1+\cdots + e_k,\quad 1\leq k\leq n,$$
where $e_i$ is the $i$th standard rank-one projection.
This implies that
$$W^*e_iW=e_i,\quad 1\leq i\leq n.$$

Since $W^*AW=B$, it follows that 
$$W^*e_iAe_iW=e_iBe_i,\quad 1\leq i\leq n,$$
and hence $$a_i=b_i,\quad 1\leq i\leq n.$$
A similar argument shows that  
$c_i=d_i,$ $1\leq i\leq n.$
\end{proof}

\begin{lem}\label{cont-field}
Let $Z$ be a second countable locally compact Hausdorff space, and let $S$ be a sub-C*-algebra of $\mathrm{M}_n(\mathrm{C_0}(Z))$. Suppose that there exist a topological space $\Delta$ and a surjective continuous map $\xi: Z\to\Delta$ such that 
\begin{enumerate}
\item\label{cond-sep} $\xi(x_1)=\xi(x_2)$ if and only if $\pi_{x_1}|_S$ is unitarily equivalent to $\pi_{x_2}|_S$, 
\item\label{cond-cont} for any $g\in S$, if $\xi(x_i)\to\xi(x)$, then $g(x_i) \to g(x)$, and
\item\label{cond-full} $\pi_x(S)=\mathrm{M}_n(\mathbb C)$, for any $x\in Z$.
\end{enumerate}
Then $S\cong\mathrm{M}_n(\mathrm{C}_0(\Delta))$.
\end{lem}
\begin{proof}
For each $f\in S$, define a function $\tilde{f}: \Delta\to \mathrm{M}_n(\Comp)$ by
$$\tilde{f}(z)=f(x),\quad\textrm{if $\xi(x)=z$}.$$
By Condition \eqref{cond-cont}, $\tilde{f}$ is well defined, and $\tilde{f}$ is continuous. Moreover, $\tilde{f}$ vanishes at infinity. As, if $z_i\in\Delta$ with $z_i\to\infty$, since $\xi$ is surjective, there are $x_i\in Z$ with $\xi(x_i)=z_i$. 
Then $x_i\to\infty$. Otherwise, there is a subsequence, say $(x_{i_k})$, converging to a point $x\in Z$. Since $\xi$ is continuous, one has that $z_{i_k}=\xi(x_{i_k})\to\xi(x)$, which contradicts the assumption $z_i\to\infty$. Hence $\tilde{f}(z_i)=f(x_i)\to 0$, and $\tilde{f}\in \mathrm{M}_n(\mathrm{C_0}(\Delta))$.

Moreover, it is clear that the map $f\to\tilde{f}$ is an injective homomorphism, and thus one can regard $S$ as a sub-C*-algebra of $\mathrm{M}_n(\mathrm{C_0}(\Delta))$. It follows from Conditions \eqref{cond-sep} and \eqref{cond-full} that $S$ is a rich sub-C*-algebra of $\mathrm{M}_n(\mathrm{C_0}(\Delta))$ in the sense of Dixmier (11.1.1 of \cite{Dix-Book}), and therefore $S=\mathrm{M}_n(\mathrm{C}_0(\Delta))$ by Proposition 11.1.6 of \cite{Dix-Book} (or, it follows from Theorem 7.2 of \cite{Kap-SCalg}).
\end{proof}

\begin{rem}
The space $\Delta$ of Lemma \ref{cont-field} is automatically locally compact Hausdorff. 
\end{rem}

\begin{thm}\label{lc-app}
Let $X$ be an infinite compact metrizable space, and let $\sigma$ be a minimal homeomorphism. Suppose that $(X, \sigma)$ has mean dimension zero. Let $$\{f_1, f_2, ..., f_n, g_1, g_2, ..., g_m\}\subseteq \mathrm{C}(X)$$ with $g_i(W)=\{0\}$, $i=1, ..., m$, for some open set $W$ containing some $y\in X$. Then, for any $\eps>0$, there is a closed neighbourhood $Y$ of $y$ contained in $W$ such that the finite subset 
$$\{f_1, f_2, ..., f_n, ug_1, ug_2, ..., ug_m\}$$
of $A_Y$, where $u$ is the canonical unitary of the crossed product, is approximated to within $\eps$ by a subhomogeneous C*-algebra $S$ in $A_Y$ with dimension ratio at most $\eps$.
\end{thm}
\begin{proof}
Let $\eps>0$ be arbitrary. Choose a finite open cover 
$$\alpha=\{U_1, U_2, ..., U_{\abs{\alpha}}\}$$
of $X$ such that
\begin{equation}\label{eq-small-nbhd}
\abs{f_i(x)-f_i(x')} < \eps\quad\mathrm{and}\quad \abs{g_j(x)-g_j(x')} < \eps,\quad x, x'\in U_i,\  1\leq i\leq\abs{\alpha}.
\end{equation}
Since $(X, \sigma)$ is minimal and has mean dimension zero, it has SBP, and therefore by Proposition \ref{small-intsct}, there is a partition of unity $\{\phi_U;\ U\in\alpha \}$ subordinate to $\alpha$ and $T\in\mathbb N$ such that 
\begin{equation}\label{sm-fctn}
\frac{1}{N}(\chi_{E}(x)+\chi_{E}(\sigma(x))+\cdots+\chi_E(\sigma^{N-1}(x)))<\frac{\eps}{\abs{\alpha}+1},\quad x\in X,\  N\geq T,
\end{equation}
where $E=\bigcup_{U\in\alpha}\phi_U^{-1}((0, 1))$.

Choose the closed neighbourhood $Y$ of $y$ in $W$ as follows: the Rokhlin partition
$$\{\{Z_1, \sigma(Z_1), ..., \sigma^{J_1-1}(Z_1)\}, ..., \{Z_k, \sigma(Z_k), ..., \sigma^{J_k-1}(Z_k)\}\}$$
corresponding as in Section \ref{Strut-AY} to $Y$ should satisfy
$$J_1\geq\max\{\frac{\abs{\alpha}+1}{\eps}, T\}.$$
By Lemma \ref{small-nbhd}, there is an open set $V$ such that $Y\subseteq V$, and 
for any $1\leq k\leq K$,
\begin{equation}\label{sm-H}
\frac{1}{J_k}(\chi_{V}(x)+\chi_{V}(\sigma(x))+\cdots+\chi_V(\sigma^{J_k-1}(x)))\leq \frac{1}{J_1}<\frac{\eps}{\abs{\alpha}+1},\quad x\in Z_{k}.
\end{equation}
Choose a continuous function $H: X\to [0, 1]$ such that
$$H^{-1}(0)=Y\quad\textrm{and}\quad H^{-1}(1)\supseteq (X\setminus V).$$
Since $Y\subseteq W$, without loss of generality, we may assume that $V\subseteq W$, and then
$$Hg_j=g_j,\quad 1\leq j\leq m.$$

Let us show that the sub-C*-algebra 
$$S:=\textrm{C*}\{\phi_U, uH;\ U\in\alpha \}\subseteq A_Y,$$
together with the closed set $Y$,
satisfies the conditions of the theorem.

For each $U\in\alpha$, pick a point $x_U\in U$. Then, by \eqref{eq-small-nbhd}, for each $f_i$, $1\leq i\leq n$, 
\begin{equation*}
\|{f_i-\sum_{U\in \alpha} f_i(x_{U})\phi_U}\| \leq  \sup_{x\in X}\sum_{U\in \alpha} \abs{f(x)-f_i(x_{U})}\phi_U(x)<\eps;
\end{equation*}
and for each $g_j$, $1\leq j\leq m$, one has
\begin{eqnarray*}
\|{ug_j-uH\sum_{U\in \alpha} g_j(x_{U})\phi_U}\| &=&\|{uHg_j-uH\sum_{U\in \alpha} g_j(x_{U})\phi_U}\|\\
&\leq &\|{g_j-\sum_{U\in \alpha} g_j(x_{U})\phi_U}\|\\
&<&\eps.
\end{eqnarray*}
This shows the approximate inclusion
$$\{f_1, f_2, ..., f_n, ug_1, ug_2, ..., ug_m\} \subseteq_\eps S.$$

Finally, let us show that $\mathrm{dimRatio}(S)<\eps$. For each $1\leq k\leq K$, consider the algebra
$$\textrm{M}_{J_k}(\mathrm{C}(\overline{Z_{k}}))$$
of Theorem \ref{Lin-Sub},
and consider the map
$$\xi_k: \overline{Z_{k}}\to\mathbb R^{(\abs{\alpha}+1)J_k-1}$$ defined by
\begin{equation}\label{defn-xi}
\xi_k(x) \mapsto ((\Phi\circ\sigma(x), \Phi\circ\sigma^2(x),..., \Phi\circ\sigma^{J_k}), (H\circ\sigma(x), ..., H\circ\sigma^{J_k-1}(x))),
\end{equation}
where the map $\Phi: \overline{Z_k}\to\Real^{\abs{\alpha}}$ is defined by $$\Phi=\bigoplus_{U\in\alpha}\phi_U.$$

By \eqref{sm-fctn} and \eqref{sm-H}, the image of $Z_k$ under $\xi_k$ is contained in the set
$$\{(t_1, t_2, ..., t_{(\abs{\alpha}+1)J_k-1})\in [0, 1]^{(\abs{\alpha}+1)J_k-1};\ \textrm{at most ${\eps} J_k$ of the $t_i$ are not $0$ or $1$}\},$$
which has dimension at most $\eps J_k-1$ (as it is a union of simplices with at most $\eps J_k$ vertices). Therefore, $\xi_k(Z_k)$ has dimension at most $\eps J_k$. For convenience,  write $\xi_k(Z_k)=\Delta_k$. We have
$$\mathrm{dim}(\Delta_k)<\eps J_k.$$

For each $x\in Z_k$, the evaluation map $\pi_x$ on $A_Y$ is an irreducible representation of $A_Y$ with dimension $J_k$. Consider the restriction of $\pi_x$ to $S$. Note that for any $x_1, x_2\in Z_{k}$, if
$$\xi_k(x_1)=\xi_k(x_2),$$
 then, by the definition of $\xi_k$, one has
$$\phi_U\circ\sigma^{i}(x_1)=\phi_U\circ\sigma^{i}(x_2)\quad\textrm{and}\quad H\circ\sigma^{j}(x_1)=H\circ\sigma^{j}(x_2),$$
where $U\in \alpha, 1\leq i\leq J_k, 1\leq j\leq J_k-1.$
By \eqref{id-01} and \eqref{id-02} of Theorem \ref{Lin-Sub}, one has that 
$$\pi_{x_1}(\phi_U)=\pi_{x_2}(\phi_U)\quad\textrm{and}\quad \pi_{x_1}(uH)=\pi_{x_2}(uH),\quad U\in \alpha.$$
Since $S$ is the sub-C*-algebra generated by $\phi_U$, $U\in {\alpha}$, and $uH$, this shows that
\begin{equation}\label{eq-rep}
\pi_{x_1}|_S=\pi_{x_2}|_S.
\end{equation}
Moreover, for any $g\in S$, $x\in Z_k$, and any sequence $(x_n)$ in $Z_k$, if $\xi_k(x_n)\to\xi_k(x)$, then 
\begin{equation}\label{thm-cond-02-pre}
\pi_{x_n}(g) \to \pi_x(g).
\end{equation}

For any $x\in Z_{k}$, the representation 
$\pi_x|_S$ is irreducible on $S$ (hence has dimension $J_k$). In fact, let us consider the image of $uH$ under $\pi_x$, which is
\begin{displaymath}
w:=\left(
\begin{array}{cccc}
0 &    &  & \\
H(\sigma(x)) & 0  &  &  \\
 & \ddots  & \ddots & \\
 &    & H(\sigma^{J_k-1}(x)) & 0
\end{array}
\right)\in \pi_x(S).
\end{displaymath}
Noting that $H^{-1}(0)=Y$ and $x\in Z_k$, we have 
\begin{equation}\label{non0}
H(\sigma^i(x))\neq 0, \quad  1\leq i \leq J_k-1.
\end{equation}
Then the C*-algebra generated by $w$ is the full matrix algebra $\mathrm{M}_{J_k}(\mathbb C)$, and the restriction of $\pi_x$ to $S$ must be irreducible. In particular,
\begin{equation}\label{thn-cond-03-pre}
\pi_x(S)=\mathrm{M}_{J_k}(\Comp).
\end{equation}
Therefore, the dimension of an irreducible representation of $S$ must be $J_k$ for some $k$, and each irreducible representation of $S$ with dimension $J_k$ is the restriction of $\pi_x$ for some $x\in Z_{k}$. 

Let $x_1, x_2\in Z_k$. 
Let us show that 
\begin{equation}\label{lem-cond-01-pre}
\textrm{$\pi_{x_1}|_S$ and $\pi_{x_2}|_S$ are unitarily equivalent if and only if $\xi_k(x_1)=\xi_k(x_2)$}.
\end{equation}

If $\xi_k(x_1)=\xi_k(x_2)$, then, as shown above, 
$$\pi_{x_1}|_{S}=\pi_{x_2}|_S.$$
In particular, $\pi_{x_1}|_S$ and $\pi_{x_2}|_S$ are unitarily equivalent.

Now, assume that $\pi_{x_1}|_S$ and $\pi_{x_2}|_S$ are unitarily equivalent. Pick $\phi_U$, and consider the pair $(\phi_U, uH)$. Again, by  \eqref{id-01} and \eqref{id-02} of Theorem \ref{Lin-Sub}, we have 
\begin{displaymath}
\pi_{x_1}(\phi_U)=
\left(
\begin{array}{ccc}
\phi_U(\sigma(x_1)) & & \\
 & \ddots &\\
 & & \phi_U(\sigma^{J_k}(x_1))
\end{array}
\right),\ \
\pi_{x_1}(uH)=
\left(
\begin{array}{cccc}
0 & & & \\
H(\sigma(x_1))& 0& & \\
&  \ddots& \ddots &\\
& & H(\sigma^{J_k-1}(x_1)) & 0 
\end{array}
\right),
\end{displaymath}
and
\begin{displaymath}
\pi_{x_2}(\phi_U)=
\left(
\begin{array}{ccc}
\phi_U(\sigma(x_2)) & & \\
 & \ddots &\\
 & & \phi_U(\sigma^{J_k}(x_2))
\end{array}
\right),\ \
\pi_{x_2}(uH)=
\left(
\begin{array}{cccc}
0 & & & \\
H(\sigma(x_2))& 0& & \\
&  \ddots& \ddots &\\
& & H(\sigma^{J_k-1}(x_2)) & 0 
\end{array}
\right).
\end{displaymath}
Since $\pi_{x_1}$ and $\pi_{x_2}$ are assumed to be unitarily equivalent, the pair of matrices $(\pi_{x_1}(\phi_U), \pi_{x_1}(uH))$ is unitarily equivalent to the pair of matrices $(\pi_{x_2}(\phi_U), \pi_{x_2}(uH))$. By \eqref{non0}, we may apply Lemma \ref{uni-sep} to obtain
$$\phi_U(\sigma^i(x_1))= \phi_U(\sigma^i(x_2))\quad\textrm{and}\quad H(\sigma^j(x_1))= H(\sigma^j(x_2)),\quad 1\leq i\leq J_k,\ 1\leq j\leq J_k-1.$$ Applying this argument for all $U\in \alpha$, we have that
$$\phi_U(\sigma^i(x_1))= \phi_U(\sigma^i(x_2))\quad\textrm{and}\quad H(\sigma^j(x_1))= H(\sigma^j(x_2)),\quad U\in\alpha,\ 1\leq i\leq J_k,\ 1\leq j\leq J_k-1,$$
and this implies (by the construction of the map $\xi_k$; see \eqref{defn-xi})
$$\xi_k(x_1)=\xi_k(x_2).$$
This proves the assertion.

Since any irreducible representation of $S$ is contained in a irreducible representation of $A_Y$, and $\{\pi_x;\  x\in Y\}$ are all of the irreducible representations of $A_Y$, the dimensions of the irreducible representations of $S$ have to be $J_1, J_2, ..., J_K$, and the map $\xi$ induces a bijection between $\mathrm{Prim}_{J_k}(S)$ and $\Delta_k$ for each $1\leq k\leq K$. Then the subquotient with $J_k$-dimensional representations of $S$, denoted by $S_k$, is a sub-C*-algebra of the subquotient with $J_k$-dimensional representations of $A_Y$, which is canonically isomorphic to $\mathrm{M}_{J_k}(\mathrm{C}_0(Z_k))$. By \eqref{lem-cond-01-pre}, for any $x_1, x_2\in Z_k$,
\begin{equation}\label{lem-cond-01}
\textrm{$\pi_{x_1}|_{S_k}$ is unitarily equivalent to $\pi_{x_2}|_{S_k}$ if and only if $\xi_k(x_1)=\xi_k(x_2)$}.
\end{equation}
By \eqref{thm-cond-02-pre}, one has that 
for any $g\in S_k$, any $x\in Z_k$, and any sequence $(x_n)$ in $Z_k$, if $\xi_k(x_n)\to\xi_k(x)$,  
\begin{equation}\label{lem-cond-02}
\pi_{x_n}(g) \to \pi_x(g).
\end{equation}
Therefore, the conditions of Lemma \ref{cont-field} are satisfied for the sub-C*-algebra $S_k$ of $\mathrm{M}_{J_k}(\mathrm{C}_0(Z_k))$, and it follows that
$$S_k\cong\mathrm{M}_{J_k}(\mathrm{C}_0(\Delta_k)).$$
This implies that  $$\mathrm{Prim}_{J_k}(S)=\mathrm{Prim}(S_k)=\Delta_k,$$ and hence
$$\mathrm{dim}(\mathrm{Prim}_{J_k}(S))= \mathrm{dim}(\Delta_k)<\eps J_k.$$ 
In particular, 
$$\textrm{dimRatio}(S)<\eps,$$
as desired.
\end{proof}


\begin{thm}\label{lc-app-w}
Let $X$ be an infinite metrizable compact space, and let $\sigma: X\to X$ be a minimal homeomorphism. If $(X, \sigma)$ has mean dimension zero, then the C*-algebra $A_y$ is a locally approximately subhomogeneous C*-algebra with slow dimension growth.
\end{thm}
\begin{proof}
This follows directly from Theorem \ref{lc-app}.
\end{proof}

\begin{thm}\label{Z-stb}
Let $X$ be an infinite metrizable compact space, and let $\sigma: X\to X$ be a minimal homeomorphism. If $(X, \sigma)$ has mean dimension zero, then the C*-algebra $A=\mathrm{C}(X)\rtimes_{\sigma}\mathbb Z$ absorbs the Jiang-Su algebra $\mathcal Z$ tensorially.
\end{thm}
\begin{proof}
By Theorem \ref{lc-app-w}, the C*-algebra $A_y$ is locally approximated by subhomogeneous C*-algebras with arbitrarily small dimension ratio. By Lemma 5.8 of \cite{Niu-MD}, the C*-algebra $A_y$ has strict comparison of positive elements. On using Lemma 5.10 of \cite{Niu-MD} in place of Theorem 3.4 of \cite{Toms-SDG}, the same argument as that of Theorem 1.2 of \cite{Toms-SDG} shows that the Cuntz semigroups of $A_y$ and $A_y\otimes\mathcal Z$ are isomorphic, and therefore $A_y\cong A_y\otimes\mathcal Z$ by Corollary 7.4 of \cite{Winter-Z-stable-02}. Since, by Theorem \ref{LALG} (see \cite{A-NCP-LAlg}), $A_y$ is centrally large in $A$ in the sense of D.~Archey and N.~C.~Phillips, by Theorem \ref{A-NCP-Z} (see \cite{A-NCP-LAlg-Z}) the nuclear C*-algebra $A$ also satisfies $A\cong A\otimes\mathcal Z$.
\end{proof}


\begin{rem}
Since this paper was first posted on arXiv, it was shown in \cite{ENST-ASH} that the C*-algebra $\mathrm{C}(X)\rtimes_\sigma\Int$ is always rationally locally approximately subhomogeneous for a minimal dynamical system $(X, \sigma)$, and the Jiang-Su stable rationally approximately subhomogeneous C*-algebras have also been classified (they were shown to have finite nuclear dimension in \cite{ENST-ASH}, and hence to be classified in \cite{EGLN-ASH} (using \cite{GLN-TAS}); this classification result was vastly generalized in  \cite{EN-K0-Z} and \cite{EGLN-Dr}). (The special case of crossed products was also considered in  \cite{Lin-Dyn} and \cite{Karen-sphere}; in \cite{Lin-Dyn} the classification of the Jiang-Su stabilizations of the crossed products---in particular, as a consequence of the present paper, the mean dimension zero crossed products themselves---was obtained, using \cite{ENST-ASH} and \cite{GLN-TAS}.) In other words, one has the following corollary. 
\end{rem}

\begin{cor}
Let $X$ be an infinite metrizable compact space, and let $\sigma: X\to X$ be a minimal homeomorphism.
If $(X, \sigma)$ has mean dimension zero, then $\mathrm{C}(X)\rtimes_\sigma\Int$ is classifiable and in fact is an ASH algebra. In particular, if the minimal system $(X, \sigma)$ has finite entropy or countably many ergodic measures, the C*-algebra $\mathrm{C}(X)\rtimes_\sigma\Int$ is classifiable and is an ASH algebra.
\end{cor}

\section{Tensor products}

In this section, let us show that the tensor product of the crossed product C*-algebras of two or more minimal homeomorphisms is  $\mathcal Z$-stable (Theorem \ref{prod-prod}). In particular, this implies that the Toms growth rank (\cite{Toms-GR}) of any crossed product C*-algebra $\mathrm{C}(X)\rtimes_\sigma\Int$ with $(X, \sigma)$ minimal is at most two. It also shows that the examples of Giol and Kerr (\cite{GK-Dyn}) are prime among the C*-algebras of  minimal homeomorphisms.

\begin{thm}\label{lc-app-mdA}
Let $X$ be an infinite compact metrizable space, and let $\sigma$ be a minimal homeomorphism. Let $$\{f_1, f_2, ..., f_n, g_1, g_2, ..., g_m\}\subseteq \mathrm{C}(X),$$ with $g_i(W)=\{0\}$, $i=1, ..., m$, for some open set $W$ containing $y$. Then, for any $\eps>0$, there is $R>0$ such that for any $J\in\mathbb{N}$, there is a closed neighbourhood $Y$ of $y$ contained in $W$ such that the finite subset 
$$\{f_1, f_2, ..., f_n, ug_1, ug_2, ..., ug_m\}$$
of $A_Y$, where $u$ is the canonical unitary of the crossed product, is approximated to within $\eps$ by a subhomogeneous C*-algebra $S$ in $A_Y$ with dimension ratio at most $R$, and with the dimension of each irreducible representation at least $J$.
\end{thm}

\begin{proof}
The proof is a slight modification of the proof of Theorem \ref{lc-app}.

Let $\eps>0$ be arbitrary. Choose a finite open cover 
$$\alpha=\{U_1, U_2, ..., U_{\abs{\alpha}}\}$$
of $X$ such that
\begin{equation}\label{eq-small-nbhd-1}
\abs{f_i(x)-f_i(y)} < \eps\quad\mathrm{and}\quad \abs{g_j(x)-g_j(y)} < \eps,\quad x, y\in U_i,\  1\leq i\leq\abs{\alpha}.
\end{equation}
Then $$R=\abs{\alpha}+1$$
is the desired constant.

Let $J \in\mathbb N$ be arbitrary. Choose the closed neighbourhood $Y$ of $y$ in $W$ as follows: the Rokhlin partition
$$\{\{Z_1, \sigma(Z_1), ..., \sigma^{J_1-1}(Z_1)\}, ..., \{Z_k, \sigma(Z_k), ..., \sigma^{J_k-1}(Z_k)\}\}$$
corresponding as in Section \ref{Strut-AY} to $Y$ should satisfy
$$J_1\geq J.$$

Pick an open set  $V$ such that $Y\subseteq V\subseteq W$, and pick a continuous function $H: X\to [0, 1]$ such that
$$H^{-1}(0)=Y\quad\textrm{and}\quad H^{-1}(1)\supseteq (X\setminus V).$$
Since $Y\subseteq W$, without loss of generality, we may assume that $V\subseteq W$, and then
$$Hg_j=g_j,\quad 1\leq j\leq m.$$

Choose a partition of unity $\{\phi_U: U\in\alpha\}$ subordinate to $\alpha$. 

Let us show that the sub-C*-algebra 
$$S:=\textrm{C*}\{\phi_U, uH;\ U\in\alpha \}\subseteq A_Y,$$
together with the closed set $Y$,
satisfies the conditions of the theorem (for $R$ and $J$).

For each $U\in\alpha$, pick a point $x_U\in U$. Then, by \eqref{eq-small-nbhd-1}, for each $f_i$, $1\leq i\leq n$,   
\begin{equation*}
\|{f_i-\sum_{U\in \alpha} f_i(x_{U})\phi_U}\| \leq  \sup_{x\in X}\sum_{U\in \alpha} \abs{f(x)-f_i(x_{U})}\phi_U(x)<\eps;
\end{equation*}
and for each $g_j$, $1\leq j\leq m$, 
\begin{eqnarray*}
\|{ug_j-uH\sum_{U\in \alpha} g_j(x_{U})\phi_U}\| & = & \|{uHg_j-uH\sum_{U\in \alpha} g_j(x_{U})\phi_U}\|\\
&\leq &\|{g_j-\sum_{U\in \alpha} g_j(x_{U})\phi_U}\|\\
&<&\eps.
\end{eqnarray*}
This shows that
$$\{f_1, f_2, ..., f_n, ug_1, ug_2, ..., ug_m\} \subseteq_\eps S.$$

Let us show that $\mathrm{dimRatio}(S)\leq R$. For each $1\leq k\leq K$, consider the algebra
$$\textrm{M}_{J_k}(\mathrm{C}(\overline{Z_{k}}))$$
of Theorem \ref{Lin-Sub},
and consider the map
$$\xi_k: \overline{Z_{k}}\to\mathbb R^{(\abs{\alpha}+1)J_k-1}$$ defined by
\begin{equation}\label{defn-xi-non-md0}
\xi_k(x) \mapsto ((\Phi\circ\sigma(x), \Phi\circ\sigma^2(x),..., \Phi\circ\sigma^{J_k}), (H\circ\sigma(x), ..., H\circ\sigma^{J_k-1}(x))),
\end{equation}
where the map $\Phi: \overline{Z_k}\to\Real^{\abs{\alpha}}$ is defined by $$\Phi=\bigoplus_{U\in\alpha}\phi_U.$$
It is clear that image of $\xi_k(Z_k)$ has dimension at most $(\abs{\alpha}+1)J_k-1$. Then an argument similar to that of Theorem \ref{lc-app} shows that the irreducible representations of $S$ have dimension
$$J_1<J_2<\cdots<J_K,$$
and that
$$\mathrm{dim}(\mathrm{Prim}_{J_k}(S))\leq (\abs{\alpha}+1)J_k-1.$$
Therefore,
$$\mathrm{dimRatio}(S)< \abs{\alpha}+1=R,$$
and the dimension of each irreducible representation of $S$ is at least $J$ (note that $J\leq J_1$).
\end{proof}

\begin{lem}\label{disj-prod}
Let $C$ and $S$ be subhomogeneous C*-algebras, and let $m_0, n_0, m_1, n_1$, and $d$ be natural numbers satisfying $m_0m_1=n_0n_1=d$ and $m_0\neq n_0$. Assume that $C$ has irreducible representations of dimensions $m_0$ and $n_0$, and that $S$ has irreducible representations of dimensions $m_1$ and $n_1$. Consider the subsets $E$ and $F$ of $X:=\mathrm{Prim}_d(C\otimes S)$ defined by
$$E=\{\rho: \rho=\pi_0\otimes\pi_1,\ \pi_0\in \mathrm{Prim}_{m_0}(C),\ \pi_1\in \mathrm{Prim}_{m_1}(S)\}$$
and
$$F=\{\rho: \rho=\pi_0\otimes\pi_1,\ \pi_0\in \mathrm{Prim}_{n_0}(C),\ \pi_1\in \mathrm{Prim}_{n_1}(S)\}.$$
Then the closures of $E$ and $F$ (in $X$) are disjoint. In particular, the sets $E$ and $F$ are relatively closed subsets of $X$.
\end{lem}

\begin{proof}
Assuming the contrary, there would be $(\pi_k^{0}\otimes\pi_k^{1})$ converging to $\pi_\infty^{(0)}\otimes \pi_\infty^{(1)}$ in $\mathrm{Prim}_d(C\otimes S)$, where
$$\pi_k^0\in\mathrm{Prim}_{m_0}(C),\ \pi_k^1\in\mathrm{Prim}_{m_1}(S),\quad k=1, 2, ...$$
and
$$\pi_\infty^0\in\mathrm{Prim}_{n_0}(C),\ \pi_\infty^1\in\mathrm{Prim}_{n_1}(S),\quad k=1, 2, ...\ .$$
Without loss of generality, we may assume that $m_0<n_0$ (hence $m_1>n_1$).

For any $c\otimes s\in C\otimes S$, consider the sequence
$$\mathrm{Tr}((\pi_k^{0}\otimes\pi_k^{1})(c\otimes s))=\mathrm{Tr}(\pi_k^{0}(c)\otimes\pi_k^{1}(s))=\mathrm{Tr}(\pi_k^{0}(c))\cdot \mathrm{Tr}(\pi_k^{1}(s)),\quad k=1, 2, ...$$
Since $(\pi_k^{0}\otimes\pi_k^{1}) \to \pi_\infty^{(0)}\otimes \pi_\infty^{(1)}$ in $\mathrm{Prim}_d(C\otimes S)$, we have
$$\mathrm{Tr}(\pi_k^{0}(c))\cdot \mathrm{Tr}(\pi_k^{1}(s))\to \mathrm{Tr}(\pi_\infty^{0}(c))\cdot \mathrm{Tr}(\pi_\infty^{1}(s)),\quad k\to\infty.$$
Setting $s=1_S$, one has that 
\begin{equation}\label{tr-conv}
\mathrm{Tr}(\pi_k^{0}(c))\to \frac{n_1}{m_1}\cdot \mathrm{Tr}(\pi_\infty^{0}(c)),\quad k\to\infty,\ c\in C.
\end{equation}
Note that $(\pi_k^0)\subseteq \mathrm{Prim}_{m_0}(C)$, $\pi_\infty^0\in \mathrm{Prim}_{n_0}(C)$, and $m_0<n_0$. There is $c\in C$ such that 
$$\pi_\infty^0(c)\neq 0\quad\textrm{but}\quad \pi(c)=0,\quad \pi\in\mathrm{Prim}_{m_0}(C).$$ In particular,
$$\pi_k^0(c)=0,\quad k=1, 2, ...\ .$$
But this contradicts to \eqref{tr-conv}.
\end{proof}

\begin{lem}\label{dimR-prod}
Let $C$ and $S$ be unital subhomogeneous C*-algebras, and let $J$ be a natural number such that each irreducible representation of $C$ or $S$ has dimension at least $J$. Then
$$\mathrm{dimRatio}(C\otimes S)\leq \frac{\mathrm{dimRatio}(C)+\mathrm{dimRatio}(S)}{J}.$$
\end{lem}
\begin{proof}
Let $d$ be any natural number. Then
$$\mathrm{Prim}_d(C\otimes S)=\bigsqcup_{mn=d}(\mathrm{Prim}_{m}(C) \times \mathrm{Prim}_n(S)).$$
By Lemma \ref{disj-prod}, each $\mathrm{Prim}_{m}(C) \times \mathrm{Prim}_n(S)$ is relatively close in $\mathrm{Prim}_d(C\otimes S)$, and then we have
\begin{eqnarray*}
\mathrm{dim}(\mathrm{Prim}_d(C\otimes S)) &=&\max_{mn=d}\{\mathrm{dim}(\mathrm{Prim}_m(C)\times\mathrm{Prim}_n(S))\}\\
&\leq&\max_{mn=d}\{\mathrm{dim}(\mathrm{Prim}_m(C))+\mathrm{dim}(\mathrm{Prim}_n(S))\}.
\end{eqnarray*}
This implies
\begin{eqnarray*}
\frac{\mathrm{dim}(\mathrm{Prim}_d(C\otimes S))}{d} & \leq & \max_{d=mn}\{\frac{\mathrm{dim}(\mathrm{Prim}_m(C))+\mathrm{dim}(\mathrm{Prim}_n(S))}{d}\} \\
&=& \max_{d=mn}\{\frac{\mathrm{dim}(\mathrm{Prim}_m(C))}{m}\cdot\frac{1}{n}+\frac{\mathrm{dim}(\mathrm{Prim}_n(S))}{n}\cdot\frac{1}{m}\} \\
&\leq & \max_{d=mn}\{\mathrm{dimRatio}(C)\cdot\frac{1}{n}+\mathrm{dimRatio}(S)\cdot\frac{1}{m}\}\\
&\leq & \frac{\mathrm{dimRatio}(C)+\mathrm{dimRatio}(S)}{J},
\end{eqnarray*}
as desired.
\end{proof}

\begin{lem}\label{prod-sgd}
Let $A$ and $B$ be C*-algebras satisfying the following: For any finite subset $\mathcal F$ of $A$ (or $B$) and any $\eps>0$, there is $R>0$ (which depends on $\mathcal F$ and $\eps$) and a sequence of unital sub-C*-algebras $(S_n)$ such that 
\begin{enumerate}
\item each $S_n$ is a subhomogeneous C*-algebra with $\mathrm{dimRatio}(S_n)\leq R$,
\item each $S_n$ approximately contains $\mathcal F$ up to $\eps$, and
\item $d_n\to\infty$ as $n\to\infty$, where $d_n$ is the smallest dimension of the irreducible representations of $S_n$.
\end{enumerate} 
Then $A\otimes B$ can be locally approximated by subhomogeneous C*-algebras with slow dimension growth.
\end{lem}
\begin{proof}
It is enough to show that for any finite subsets $\F\subseteq A$, $\mathcal G\subseteq B$, and any $\eps\in(0, 1)$, there is a subhomogeneous C*-algebra $D$ in $A\otimes B$ such that $\mathcal F\otimes \mathcal G\subseteq_\eps D$ and $\mathrm{dimRatio}(D)<\eps$.

Without loss of generality, we may assume that each element of $\mathcal F$ and $\mathcal G$ has norm one. By the assumptions, there are subhomogeneous C*-algebras $C\subseteq A$ and $S\subseteq B$ such that
$$\mathrm{dimRatio}(C)\leq R\quad\textrm{and}\quad \F\subseteq_{\frac{\eps}{4}} C,$$
and
$$\mathrm{dimRatio}(S)\leq R\quad\textrm{and}\quad \mathcal G\subseteq_{\frac{\eps}{4}} S,$$
and, furthermore, the dimension of each irreducible representation of $C$ or $S$ is at least $\frac{2R}{\eps}$. Then consider the C*-algebra
$$D:=C\otimes S.$$

By Lemma \ref{dimR-prod}, 
$$\mathrm{dimRatio}(D)\leq \eps.$$
A straightforward calculation also shows that
$$\mathcal F\otimes\mathcal G\subseteq_{\eps} D,$$
and this finishes the proof.
\end{proof}

\begin{prop}\label{prod-Lalg}
Let $(X_1, \sigma_1)$ and $(X_2, \sigma_2)$ be minimal systems, where $X_1$ and $X_2$ are infinite. Fix $y_1\in X_1$ and $y_2\in X_2$, and consider the large sub-C*-algebras $$A_{y_1}\subseteq \mathrm{C}(X_1)\rtimes_{\sigma_1}\Int \quad \textrm{and}\quad A_{y_2}\subseteq \mathrm{C}(X_2)\rtimes_{\sigma_2}\Int.$$
Then $$A_{y_1}\otimes A_{y_2}\cong (A_{y_1}\otimes A_{y_2})\otimes\mathcal Z.$$
\end{prop}
\begin{proof}
By Theorem \ref{lc-app-mdA} and Lemma \ref{prod-sgd}, the C*-algebra $A_{y_1}\otimes A_{y_2}$ is locally approximated by subhomogeneous C*-algebras with slow dimension growth, and therefore (with the same argument as that of Theorem \ref{Z-stb}) it absorbs the Jiang-Su algebra tensorially.
\end{proof}

\begin{thm}\label{prod-prod}
Let $(X_1, \sigma_1)$ and $(X_2, \sigma_2)$ be minimal systems, where $X_1$ and $X_2$ are infinite metrizable compact  spaces. Consider the crossed-product C*-algebras $$A_1=\mathrm{C}(X_1)\rtimes_{\sigma_1}\Int \quad \textrm{and}\quad A_{2}= \mathrm{C}(X_2)\rtimes_{\sigma_2}\Int.$$
Then $$A_{1}\otimes A_{2}\cong (A_{1}\otimes A_{2})\otimes\mathcal Z.$$
In particular, the (crossed product) C*-algebra of a single minimal homeomorphism has Toms growth rank (\cite{Toms-GR}) at most two. 
\end{thm}
\begin{proof}
By Proposition \ref{prod-Lalg}, the C*-algebra $A_{y_1}\otimes A_{y_2}$ absorbs the Jiang-Su algebra $\mathcal Z$. By Lemma \ref{CLalg-prod}, $A_{y_1}\otimes A_{y_2}$ is centrally large in $A_1\otimes A_2$. By Theorem \ref{A-NCP-Z} (see 2.3 of \cite{A-NCP-LAlg-Z}), the nuclear C*-algebra $A_1\otimes A_2$ absorbs the Jiang-Su algebra.
\end{proof}

\begin{cor}
Let $(X_1, \sigma_1)$ and $(X_2, \sigma_2)$ be minimal systems, where $X_1$ and $X_2$ are infinite metrizable compact spaces. Consider the crossed-product C*-algebras $$A_1=\mathrm{C}(X_1)\rtimes_{\sigma_1}\Int \quad \textrm{and}\quad A_{2}= \mathrm{C}(X_2)\rtimes_{\sigma_2}\Int.$$
Then $A_{1}\otimes A_{2}$ is classifiable and is an ASH algebra.
\end{cor}
\begin{proof}
By \cite{ENST-ASH}, the algebras $A_{1}$ and $A_{2}$ are rationally locally approximately subhomogeneous, and therefore $A_1\otimes A_2$ is rationally locally approximately subhomogeneous. By Theorem \ref{prod-prod}, $A_1\otimes A_2$ is Jiang-Su stable, and hence has finite decomposition rank by \cite{ENST-ASH} again and is classifiable by \cite{EN-K0-Z} and \cite{EGLN-Dr} (alternatively, the classifiability of $A_1\otimes A_2$ follows from \cite{EGLN-ASH}---which of course also uses \cite{ENST-ASH}).
\end{proof}

\begin{rem}
Note that, with $A_1$ and $A_2$ as in \ref{prod-prod}, $$A_1\otimes A_2\cong \mathrm{C}(X_1\times X_2)\rtimes_{\sigma}\Int^2,$$
where the action $\sigma$ of $\Int^2$ on $X_1 \times X_2$ is the product action: $$\sigma_{(m, n)}(x_1, x_2)=(\sigma_1^m(x_1), \sigma_2^n(x_2)).$$ Thus, Theorem \ref{prod-prod} states that for minimal actions of $\mathbb Z$ on $X_1$ and $X_2$ the crossed product C*-algebra $\mathrm{C}(X_1\times X_2)\rtimes_{\sigma}\Int^2$ always absorbs the Jiang-Su algebra.

On the other hand, the minimal $\Int^2$-system $(X_1\times X_2, \sigma)$ has mean dimension zero, as shown below. Therefore, Theorem \ref{prod-prod} is evidence that the C*-algebra of a minimal free action of $\Int^n$ (or an arbitrary amenable group) with mean dimension zero should be classifiable.

Let us show that $(X_1\times X_2, \sigma)$ has mean dimension zero:
Let $\alpha$ be an open cover of $X_1\times X_2$. Pick open covers $\beta_1$, $\beta_2$ of $X_1$, $X_2$, respectively, such that $$\beta:= \{U\times V: U\in\beta_1,\ V\in\beta_2 \}=\pi_1^{-1}(\beta_1)\vee\pi_2^{-1}(\beta_2)$$ is a refinement of $\alpha$, where $\pi_i: X_1\times X_2 \to X_i$, $i=1, 2$, is the projection map.


For each $M, N\in\mathbb N$, consider $$F_{M, N}:=\{(m, n): 0\leq m\leq M-1, 0\leq n\leq N-1\}.$$
Note that for any $0\leq m< M$ and $0\leq n< N$, one has
$$\sigma_{(m, n)}(\beta) = \{\sigma_1^m(U)\times \sigma_2^n(V): U\in\beta_1,\ V\in\beta_2 \}=\pi^{-1}_1(\sigma_1^m(\beta_1))\vee\pi^{-1}_2(\sigma_2^n(\beta_2)),$$
and
\begin{eqnarray*}
\bigvee_{(m, n)\in F_{M, N}}\sigma^{(m, n)}(\beta) &=& \bigvee_{m=0}^{M-1}\bigvee_{n=0}^{N-1} (\pi_1^{-1} (\sigma_1^m(\beta_1))\vee\pi_2^{-1}(\sigma_2^n(\beta_2))) \\
&=& (\bigvee_{m=0}^{M-1} \pi_1^{-1}(\sigma_1^m(\beta_1)))\vee(\bigvee_{n=0}^{N-1} \pi_2^{-1}(\sigma_2^n(\beta_2))).
\end{eqnarray*}
Therefore
\begin{eqnarray*}
\mathcal D(\bigvee_{(m, n)\in F_{M, N}}\sigma^{(m, n)}(\beta)) & = & \mathcal D((\bigvee_{m=1}^M \pi_1^{-1}(\sigma_1^m(\beta_1)))\vee(\bigvee_{n=1}^N \pi_2^{-1}(\sigma_2^n(\beta_2))))\\
&\leq & \sum_{m=0}^{M-1} \mathcal D( \pi_1^{-1}(\sigma_1^m(\beta_1))) + \sum_{n=0}^{N-1} \mathcal D( \pi_2^{-1}(\sigma_1^n(\beta_2))) \\
&\leq & M \mathcal D(\beta_1) +N \mathcal D(\beta_2),
\end{eqnarray*}
and
$$\lim_{M, N\to\infty}\frac{1}{\abs{F_{M, N}}}\mathcal D(\bigvee_{(m, n)\in F_{M, N}}\sigma^{(m, n)}(\beta))\leq \lim_{M, N\to\infty}\frac{M \mathcal D(\beta_1) +N \mathcal D(\beta_2)}{MN}=0.$$

Since $\beta$ is a refinement of $\alpha$, one has 
$$\mathcal D(\bigvee_{(m, n)\in F_{M, N}}\sigma^{(m, n)}(\alpha))\leq \mathcal D(\bigvee_{(m, n)\in F_{M, N}}\sigma^{(m, n)}(\beta))$$
and
$$\lim_{M, N\to\infty}\frac{1}{\abs{F_{M, N}}}\mathcal D(\bigvee_{(m, n)\in F_{M, N}}\sigma^{(m, n)}(\alpha))\leq \lim_{M, N\to\infty}\frac{1}{\abs{F_{M, N}}}\mathcal D(\bigvee_{(m, n)\in F_{M, N}}\sigma^{(m, n)}(\beta))=0.$$
Hence $$\mathrm{mdim}(X_1 \times X_2, \sigma)=0.$$
\end{rem}

%
%
%

%
%

\bibliographystyle{plainurl}

\begin{thebibliography}{10}

\bibitem{A-NCP-LAlg}
D.~Archey and N.~C. Phillips.
\newblock Permanence of stable rank one for centrally large subalgebras and
  crossed products by minimal homeomorphisms.
\newblock 05 2015.
\newblock URL: \url{http://arxiv.org/abs/1505.00725}, \href
  {http://arxiv.org/abs/1505.00725} {\path{arXiv:1505.00725}}.
  
  \bibitem{A-NCP-LAlg-Z}
D.~Archey, J.~Buck, and N.~C. Phillips.
\newblock Centrally large subalgebras and tracially {$\mathcal Z$}-absorption.
\newblock {\em Preprint}.

\bibitem{Dix-Book}
J.~Dixmier.
\newblock {\em {C*}-algebras}, volume~15 of {\em North-Holland Mathematical
  Library}.
\newblock North-Holland Publishing Co., Amsterdam-New York-Oxford, 1977.

\bibitem{EGLN-ASH}
G.~A. Elliott, G.~Gong, H.~Lin, and Z.~Niu.
\newblock The classification of simple separable unital $\mathcal Z$-stable local {ASH}-algebras.
\newblock {\em Submitted}, 06 2015.
\newblock URL: \url{http://arxiv.org/abs/1506.02308}, \href
  {http://arxiv.org/abs/1506.02308} {\path{arXiv:1506.02308}}.

\bibitem{EGLN-Dr}
G.~A. Elliott, G.~Gong, H.~Lin, and Z.~Niu.
\newblock On the classification of simple amenable {C*}-algebras with finite
  decomposition rank, {II}.
\newblock {\em Submitted}, 07 2015.
\newblock URL: \url{http://arxiv.org/abs/1507.03437}, \href
  {http://arxiv.org/abs/1507.03437} {\path{arXiv:1507.03437}}.

\bibitem{EN-K0-Z}
G.~A. Elliott and Z.~Niu.
\newblock On the classification of simple amenable {C*}-algebras with finite
  decomposition rank.
\newblock In {\em ``Operator Algebras and their Applications: A tribute to
  Richard V.~Kadison", Contemporary Mathematics}. Amer. Math. Soc., 2015, in
  press.
\newblock \href {http://arxiv.org/abs/1507.07876} {\path{arXiv:1507.07876}}.

\bibitem{ENST-ASH}
G.~A. Elliott, Z.~Niu, L.~Santiago, and A.~Tikuisis.
\newblock Decomposition rank of approximately subhomogeneous {C*}-algebras.
\newblock {\em Submitted}, 05 2015.
\newblock URL: \url{http://arxiv.org/abs/1505.06100}, \href
  {http://arxiv.org/abs/1505.06100} {\path{arXiv:1505.06100}}.

\bibitem{GK-Dyn}
J.~Giol and D.~Kerr.
\newblock Subshifts and perforation.
\newblock {\em J. Reine Angew. Math.}, 639:107--119, 2010.
\newblock URL: \url{http://dx.doi.org/10.1515/CRELLE.2010.012}, \href
  {http://dx.doi.org/10.1515/CRELLE.2010.012}
  {\path{doi:10.1515/CRELLE.2010.012}}.

\bibitem{GLN-TAS}
G.~Gong, H.~Lin, and Z.~Niu.
\newblock Classification of finite simple amenable {$\mathcal Z$}-stable
  {C*}-algebras.
\newblock {\em Submitted}, 01 2015.
\newblock URL: \url{http://arxiv.org/abs/1501.00135}, \href
  {http://arxiv.org/abs/1501.00135} {\path{arXiv:1501.00135}}.

\bibitem{Kap-SCalg}
I.~Kaplansky.
\newblock The structure of certain operator algebras.
\newblock {\em Trans. Amer. Math. Soc.}, 70:219--255, 1951.

\bibitem{Lin-Dyn}
H.~Lin.
\newblock Crossed products and minimal dynamical systems.
\newblock 02 2015.
\newblock URL: \url{http://arxiv.org/abs/1502.06658}, \href
  {http://arxiv.org/abs/1502.06658} {\path{arXiv:1502.06658}}.

\bibitem{LP-Dym}
H.~Lin and N.~C. Phillips.
\newblock Crossed products by minimal homeomorphisms.
\newblock {\em J. Reine Angew. Math.}, 641:95--122, 2010.
\newblock URL: \url{http://dx.doi.org/10.1515/CRELLE.2010.029}, \href
  {http://dx.doi.org/10.1515/CRELLE.2010.029}
  {\path{doi:10.1515/CRELLE.2010.029}}.

\bibitem{Lin-Q-RSA}
Q.~Lin.
\newblock Analytic structure of the transformation group {C*}-algebra
  associated with minimal dynamical systems.
\newblock {\em Preprint}.

\bibitem{QL-Ph-min-diff}
Q.~Lin and N.~C. Phillips.
\newblock Direct limit decomposition for {C*}-algebras of minimal
  diffeomorphisms.
\newblock In {\em Operator algebras and applications}, volume~38 of {\em Adv.
  Stud. Pure Math.}, pages 107--133. Math. Soc. Japan, Tokyo, 2004.

\bibitem{Lind-MD}
E.~Lindenstrauss.
\newblock Mean dimension, small entropy factors and an embedding theorem.
\newblock {\em Inst. Hautes {\'E}tudes Sci. Publ. Math.}, (89):227--262 (2000),
  1999.
\newblock URL: \url{http://www.numdam.org/item?id=PMIHES_1999__89__227_0}.

\bibitem{Lindenstrauss-Weiss-MD}
E.~Lindenstrauss and B.~Weiss.
\newblock Mean topological dimension.
\newblock {\em Israel J. Math.}, 115:1--24, 2000.
\newblock URL: \url{http://dx.doi.org/10.1007/BF02810577}, \href
  {http://dx.doi.org/10.1007/BF02810577} {\path{doi:10.1007/BF02810577}}.

\bibitem{Niu-MD}
Z.~Niu.
\newblock Mean dimension and {AH}-algebras with diagonal maps.
\newblock {\em J. Funct. Anal.}, 266(8):4938--4994, 2014.
\newblock URL: \url{http://dx.doi.org/10.1016/j.jfa.2014.02.010}, \href
  {http://dx.doi.org/10.1016/j.jfa.2014.02.010}
  {\path{doi:10.1016/j.jfa.2014.02.010}}.

\bibitem{Phill-RSA1}
N.~C. Phillips.
\newblock Recursive subhomogeneous algebras.
\newblock {\em Trans. Amer. Math. Soc.}, 359(10):4595--4623 (electronic), 2007.
\newblock URL: \url{http://dx.doi.org/10.1090/S0002-9947-07-03850-0}, \href
  {http://dx.doi.org/10.1090/S0002-9947-07-03850-0}
  {\path{doi:10.1090/S0002-9947-07-03850-0}}.

\bibitem{NCP-LAlg}
N.~C. Phillips.
\newblock Large subalgebras.
\newblock 08 2014.
\newblock URL: \url{http://arxiv.org/abs/1408.5546}, \href
  {http://arxiv.org/abs/1408.5546} {\path{arXiv:1408.5546}}.

\bibitem{Put-PJM}
I.~F. Putnam.
\newblock The {C*}-algebras associated with minimal homeomorphisms of the
  {C}antor set.
\newblock {\em Pacific J. Math.}, 136(2):329--353, 1989.
\newblock URL:
  \url{http://projecteuclid.org/getRecord?id=euclid.pjm/1102650733}.

\bibitem{Karen-sphere}
K.~R. Strung.
\newblock {C*}-algebras of minimal dynamical systems of the product of a Cantor set and an odd dimensional sphere.
\newblock {\em J. Funct. Anal.}, 268(3):671--689, 2015.
\newblock URL: \url{http://dx.doi.org/10.1016/j.jfa.2014.10.014}, \href
  {http://dx.doi.org/10.1016/j.jfa.2014.10.014}
  {\path{doi:10.1016/j.jfa.2014.10.014}}.

\bibitem{Toms-GR}
A.~S. Toms.
\newblock Dimension growth for {C*}-algebras.
\newblock {\em Adv. Math.}, 213(2):820--848, 2007.
\newblock URL: \url{http://dx.doi.org/10.1016/j.aim.2007.01.011}, \href
  {http://dx.doi.org/10.1016/j.aim.2007.01.011}
  {\path{doi:10.1016/j.aim.2007.01.011}}.

\bibitem{Toms-SDG}
A.~S. Toms.
\newblock {K}-theoretic rigidity and slow dimension growth.
\newblock {\em Invent. Math.}, 183(2):225--244, 2011.
\newblock URL: \url{http://dx.doi.org/10.1007/s00222-010-0273-8}, \href
  {http://dx.doi.org/10.1007/s00222-010-0273-8}
  {\path{doi:10.1007/s00222-010-0273-8}}.

\bibitem{TW-Dym}
A.~S. Toms and W.~Winter.
\newblock Minimal dynamics and the classification of {C*}-algebras.
\newblock {\em Proc. Natl. Acad. Sci. USA}, 106(40):16942--16943, 2009.
\newblock URL: \url{http://dx.doi.org/10.1073/pnas.0903629106}, \href
  {http://dx.doi.org/10.1073/pnas.0903629106}
  {\path{doi:10.1073/pnas.0903629106}}.

\bibitem{TW-Dym-1}
A.~S. Toms and W.~Winter.
\newblock Minimal dynamics and {K}-theoretic rigidity: {E}lliott's conjecture.
\newblock {\em Geom. Funct. Anal.}, 23(1):467--481, 2013.
\newblock URL: \url{http://dx.doi.org/10.1007/s00039-012-0208-1}, \href
  {http://dx.doi.org/10.1007/s00039-012-0208-1}
  {\path{doi:10.1007/s00039-012-0208-1}}.

\bibitem{Winter-Z-stable-02}
W.~Winter.
\newblock Nuclear dimension and {$\mathcal{Z}$}-stability of pure
  {C*}-algebras.
\newblock {\em Invent. Math.}, 187(2):259--342, 2012.
\newblock URL: \url{http://dx.doi.org/10.1007/s00222-011-0334-7}, \href
  {http://dx.doi.org/10.1007/s00222-011-0334-7}
  {\path{doi:10.1007/s00222-011-0334-7}}.

\end{thebibliography}

\end{document}